\documentclass[12pt]{amsart}
\usepackage{amscd}
\usepackage{amsmath}
\usepackage{amssymb}
\usepackage{units}
\usepackage[all]{xy}

\usepackage{algorithm}
\usepackage{algorithmic}

\def\frk{\frak}               

\def\mm{{\frk m}}

\def\Phi{{\frk n}}
\def\Phi{{\frk N}}
\def\opn#1#2{\def#1{\operatorname{#2}}} 
\opn\chara{char} \opn\length{\ell} \opn\pd{pd} \opn\rk{rk}
\opn\projdim{proj\,dim} \opn\injdim{inj\,dim} \opn\rank{rank}
\opn\depth{depth} \opn\sdepth{sdepth} \opn\fdepth{fdepth}
\opn\grade{grade} \opn\height{height} \opn\embdim{emb\,dim}
\opn\codim{codim}  \opn\min{min} \opn\max{max}

\opn\Tr{Tr} \opn\bigrank{big\,rank}
\opn\superheight{superheight}\opn\lcm{lcm}
\opn\trdeg{tr\,deg}
\opn\reg{reg} \opn\lreg{lreg} \opn\ini{in} \opn\lpd{lpd}
\opn\size{size}
\opn\div{div} \opn\Div{Div} \opn\cl{cl} \opn\Cl{Cl}
\opn\Spec{Spec} \opn\Supp{Supp} \opn\supp{supp} \opn\Sing{Sing}
\opn\Ass{Ass} \opn\Min{Min}
\opn\Ann{Ann} \opn\Rad{Rad} \opn\Soc{Soc}
\opn\Im{Im} \opn\Ker{Ker} \opn\Coker{Coker} \opn\Am{Am}
\opn\Hom{Hom} \opn\Tor{Tor} \opn\Ext{Ext} \opn\End{End}
\opn\Aut{Aut} \opn\id{id}  \opn\deg{deg}

\opn\nat{nat}
\opn\pff{pf}
\opn\Pf{Pf} \opn\GL{GL} \opn\SL{SL} \opn\mod{mod} \opn\ord{ord}
\opn\Gin{Gin} \opn\Hilb{Hilb}
\opn\aff{aff} \opn\con{conv} \opn\relint{relint} \opn\st{st}
\opn\lk{lk} \opn\cn{cn} \opn\core{core} \opn\vol{vol}
\opn\link{link} \opn\star{star}
\opn\gr{gr}

\def\pot#1#2{#1[\kern-0.28ex[#2]\kern-0.28ex]}

\opn\dirlim{\underrightarrow{\lim}}
\opn\inivlim{\underleftarrow{\lim}}
\let\to=\rightarrow

\def\Implies{\ifmmode\Longrightarrow \else
        \unskip${}\Longrightarrow{}$\ignorespaces\fi}
\def\implies{\ifmmode\Rightarrow \else
        \unskip${}\Rightarrow{}$\ignorespaces\fi}
\def\iff{\ifmmode\Longleftrightarrow \else
        \unskip${}\Longleftrightarrow{}$\ignorespaces\fi}

\let\:=\colon
\newtheorem{Theorem}{Theorem}[]
\newtheorem{Lemma}[Theorem]{Lemma}
\newtheorem{Corollary}[Theorem]{Corollary}
\newtheorem{Proposition}[Theorem]{Proposition}

\newtheorem{Example}[Theorem]{Example}

\let\epsilon\varepsilon
\let\phi=\varphi
\let\kappa=\varkappa
\def\qed{\ifhmode\textqed\fi
      \ifmmode\ifinner\quad\qedsymbol\else\dispqed\fi\fi}
\def\textqed{\unskip\nobreak\penalty50
       \hskip2em\hbox{}\nobreak\hfil\qedsymbol
       \parfillskip=0pt \finalhyphendemerits=0}
\def\dispqed{\rlap{\qquad\qedsymbol}}

\opn\dis{dis}
\def\pnt{{\raise0.5mm\hbox{\large\bf.}}}

\opn\Lex{Lex}

\begin{document}

\title{  Constructive N\'eron Desingularization of algebras with big smooth locus.}

\author{Zunaira Kosar, Gerhard Pfister and Dorin Popescu }
\thanks{}

\address{Zunaira Kosar, Abdus Salam School of Mathematical Sciences,GC University, Lahore, Pakistan}
\email{ zunairakosar@gmail.com}

\address{Gerhard Pfister,  Department of Mathematics, University of Kaiserslautern, Erwin-Schr\"odinger-Str., 67663 Kaiserslautern, Germany}
\email{pfister@mathematik.uni-kl.de}

\address{Dorin Popescu, Simion Stoilow Institute of Mathematics of the Romanian Academy, Research unit 5,
University of Bucharest, P.O.Box 1-764, Bucharest 014700, Romania}
\email{dorin.popescu@imar.ro}

\begin{abstract} An algorithmic proof of the General N\'eron Desingularization theorem and its uniform version is given for morphisms with big smooth locus. This generalizes the results for the one-dimensional case (cf.  \cite{PP}, \cite{KPP}).

\end{abstract}

\maketitle

\vskip 0.5 cm

\section*{Introduction}

Motivated to generalize Artin's Approximation Theorem (cf. \cite{A}) to excellent  Henselian 
rings the third author developed a powerful tool, the General N\'eron Desingularization (cf.
\cite{P}). This result was discussed and used later by many other authors (cf. \cite{An}, \cite{S}, \cite{Sp}). The proof of the Desingularization was not constructive. In this paper 
we want to give an algorithm to compute N\'eron Desingularization for an important special case. We begin recalling some standard definitions.      

A ring morphism $u:A\to A'$ of Noetherian rings has  {\em regular fibers} if for all prime ideals $P\in \Spec A$ the ring $A'/PA'$ is a regular  ring, i.e. its localizations are regular local rings.
It has {\em geometrically regular fibers}  if for all prime ideals $P\in \Spec A$ and all finite field extensions $K$ of the fraction field of $A/P$ the ring  $K\otimes_{A/P} A'/PA'$ is regular.

A flat morphism of Noetherian rings $u$ is {\em regular} if its fibers are geometrically regular. If $u$ is regular of finite type then $u$ is called {\em smooth}. A localization of a smooth algebra is called {\em essentially smooth}. A Henselian Noetherian local ring $A$ is {\em excellent} if the completion map $A\to \hat A$ is regular.

\begin{Theorem} (General N\'eron Desingularization, Popescu \cite{P}, \cite{P1}, Andr\'e \cite{An}, Swan \cite{S}, Spivakovsky \cite{Sp})\label{gnd}  Let $u:A\to A'$ be a  regular morphism of Noetherian rings and $B$ an  $A$-algebra of finite type. Then  any $A$-morphism $v:B\to A'$   factors through a smooth $A$-algebra $C$, that is $v$ is a composite $A$-morphism $B\to C\to A'$.
\end{Theorem}
Constructive General N\'eron Desingularization for the case when the rings $A$ and $A'$ are one-dimensional local rings, is given in \cite{AP}, \cite{PP} and \cite{KPP}, the two dimensional case is partially done in \cite{PP1}. The purpose of this paper is to find a constructive proof  for the case when rings $A$ and $A'$ are of dimension $m$ and the smooth locus of $B\to A'$ is big. We proceed using induction on the dimension of rings, with the induction step given in Proposition\  \ref{proposition}. In the Section 2 we prove a uniform General N\'eron Desingularization for $m$-dimensional local Cohen-Macaulay rings and some  consequences of it. We also give an algorithm to find a uniform General N\'eron Desingularization using {\sc Singular}.

\section{Constructive  N\'eron Desingularization}

 Let $u:A\to A'$ be a flat morphism of Noetherian local rings of dimension $m$. Suppose that  the maximal ideal $\mm$ of $A$ generates the maximal ideal of $A'$, $A'$ is Henselian and $u$ is a regular morphism.

  Let $B=A[Y]/I$, $Y=(Y_1,\ldots,Y_n)$. If $f=(f_1,\ldots,f_r)$, $r\leq n$ is a system of polynomials from $I$ then we can define the ideal $\Delta_f$ generated by all $r\times r$-minors of the Jacobian matrix $(\partial f_i/\partial Y_j)$.   After Elkik \cite{El} let $H_{B/A}$ be the radical of the ideal $\sum_f ((f):I)\Delta_fB$, where the sum is taken over all systems of polynomials $f$ from $I$ with $r\leq n$.
 $H_{B/A}$ defines the non smooth locus of $B$ over $A$.
  $B$ is {\em standard smooth} over $A$ if  there exists  $f$ in $I$ as above such that $B= ((f):I)\Delta_fB$.

  The aim of this section is to give an algorithmic  proof of the following theorem.
\vskip 0.3 cm
 \begin{Theorem} \label{m} Any $A$-morphism $v:B\to A'$  such that $v(H_{B/A}A')$ is $\mm A'$-primary factors through a standard smooth $A$-algebra $B'$.
 \end{Theorem}

 To prove the above theorem we need the following proposition.

 \begin{Proposition}\label{proposition}
 Let $A$ and $A'$ be Noetherian local rings of dimension $m$ and $u:A\to A'$ be a regular morphism. Suppose that $A'$ is Henselian. Let $B=A[Y]/I$, $Y=(Y_1,\ldots,Y_n)$, $f=(f_1,\ldots,f_r)$, $r\leq n$ be a system of polynomials from $I$ as above, $(M_j)_{j\in [q]}$  some $r\times r$-minors  \footnote{ We use the notation $[q]=\{1,\ldots,q\}$.}  of the Jacobian matrix $(\partial f_i/\partial Y_{j'})$,  $(N_j)_{j \in [q]} \in  ((f):I)$ and set $P:=\sum _{j=1}^q N_jM_j$. Let  $v:B\to A'$ be an $A$-morphism. Suppose that
\begin{enumerate}
 \item{} there exist an element $d\in A$ such that $d\equiv P $ modulo $I$ and

 \item{} there exist a smooth $A$-algebra $D$ and an $A$-morphism $\omega:D\to A'$ such that $\Im v\subset \Im\omega +d^{2e+1}A'$ and for ${\bar A}=A/(d^{2e+1})$ (defining e by $(0:_Ad^e)=(0:_Ad^{e+1})$) the map $\bar{v}={\bar A} {\otimes}_{A} v: \bar{B}=B/d^{2e+1}B \to \bar{A}'=A'/d^{2e+1}A'$ factors through $\bar{D}=D/d^{2e+1}D$.
\end{enumerate}
 Then there exist a $B$-algebra $B'$ which is standard smooth over $A$ such that $v$ factors through $B'$.
 \end{Proposition}

 \begin{proof}
 Let $\delta:B\otimes_AD\cong D[Y]/ID[Y]\to A'$ be the $A$-morphism given by $b\otimes \lambda\to v(b)\omega(\lambda)$.
First we  show that
 $\delta$ factors through a special  $B\otimes_AD$-algebra $E$ of finite type.

Let the map $\bar B\to \bar D$ is given by $Y\to y'+d^{2e+1}D$. Thus $I(y')\equiv 0$ modulo $d^{2e+1}D$. Since $\bar v$ factors through $\bar \omega$ we see that $\bar \omega(y'+d^{2e+1}D)=\bar y$. Set $\tilde y=\omega(y')$. We get
  $y-\tilde y=v(Y)-\tilde y\in d^{2e+1}A'^n$, let us say $y-\tilde y=d^{e+1}\nu$ for $\nu\in d^eA'^n$.

We have $M_j=\det H_j$, where $H_j$ is the matrix $(\partial f_i/\partial Y_{j'})_{i\in [r],{j'}\in [n]}$ completed with some $(n-r)$ rows from $0,\ 1$. Since $d\equiv P$ modulo $I$ we get $P(y')\equiv d$ modulo $d^{2e+1}$ in $D$ because $I(y')\equiv 0$ modulo $d^{2e+1}D$. Thus $P(y')=ds$ for some $s\in D$ with $s\equiv 1$ modulo $d$.
 Let $G'_j$ be the adjoint matrix of $H_j$ and $G_j=N_jG'_j$. We have
$G_jH_j=H_jG_j=M_jN_j\mbox{Id}_n$
and so
$$ds\mbox{Id}_n=P(y')\mbox{Id}_n=\sum_{j=1}^{q}G_j(y')H_j(y').$$

But  $H_j$ is the matrix $(\partial f_i/\partial Y_{j'})_{i\in [r],{j'}\in [n]}$ completed with some $(n-r)$ rows from $0,\ 1$. Especially we obtain
 \begin{equation}\label{identity1}(\partial f/\partial Y){G_j}=(M_jN_j\mbox{Id}_r|0).\end{equation}

 Then ${t_j}:=H_j(y')\nu\in d^eA'^n$
satisfies
$${G}_j(y'){t_j}=M_j(y')N_j(y')\nu=d{s}\nu$$
 and so
 $${s}(y-\tilde y)=d^e\sum_{j=1}^{q}\omega({G}_j(y')){t_j}.$$
 Let
 \begin{equation}\label{def of h1}{h}={s}(Y-y')-d^e\sum_{j=1}^{q}{G}_j(y'){T_j},\end{equation}
 where  ${T_j}=({T}_1,\ldots,{T}_r, T_{j,r+1}, \ldots , T_{j,n})$ are new variables. The kernel of the map
${\phi}:D[Y,{T}]\to A'$ given by $Y\to y$, ${T_j}\to {t_j}$ contains ${h}$. Since
$${s}(Y-y')\equiv d^e\sum_{j=1}^{q}{G}_j(y')T_j\ \mbox{modulo}\ {h}$$
and
$$f(Y)-f(y')\equiv \sum_{j'}(\partial f/\partial Y_{j'})((y') (Y_{j'}-y'_{j'})$$
modulo higher order terms in $Y_{j'}-y'_{j'}$, by Taylor's formula we see that for $p=\max_i \deg f_i$ we have
\begin{equation}\label{def of Q1}{s}^pf(Y)-{s}^pf(y')\equiv  \sum_{j'}{s}^{p-1}d^e(\partial f/\partial Y_{j'})(y')\sum_{j=1}^{q} {G}_{jj'}(y'){T}_{jj'}+d^{2e}{Q}\end{equation}
modulo $h$ where ${Q}\in {T}^2 D[{T}]^r$.   We have $f(y')=d^{e+1}{b}$ for some ${b}\in d^eD^r$. Then
\begin{equation}\label{def of g1}{g}_i={s}^p{b}_i+{s}^p{T}_i+d^{e-1}{Q}_i, \qquad i\in [r] \end{equation}  is in the kernel of $\phi$. Indeed,  we have ${s}^pf_i=d^{e+1}{g}_i\ \mbox{modulo}\ {h}$ because of (\ref{def of Q1}) and $P(y')=d{s}$. Thus
$d^{e+1}\phi({g})=d^{e+1}{g}(t)\in ({h}(y,{t}),f(y))=(0)$ and ${g}({t})\in d^eA'^r$ and so $g(t)\in (0:_{A'}d^{e+1})\cap d^eA'=0$ because $(0:_{A'}d^e)=(0:_{A'}d^{e+1})$, the map $u$ being flat. Set ${E}=D[Y,{T}]/(I,{g},{h})$ and let  ${\psi}:{E}\to A'$ be the map induced by $\phi$. Clearly, $v$ factors through $\psi$ because $v$ is the composed map $B\to B\otimes_AD\cong D[Y]/I\to {E}\xrightarrow{{\psi}} A'$.

Now we show that
there exist ${s}',{s}''\in{E}$ such that ${E}_{{s}{s}'{s}''}$ is standard smooth over $A$ and $\psi$ factors through ${E}_{{s}{s}'{s}''}$.

Note that the $r\times r$-minor  ${s}'$ of $(\partial {g}/ \partial {T})$ given by the first  $r$-variables ${T}$ is from ${s}^{rp}+({T})\subset 1+(d,{T})$ because ${Q}\in ({T})^2$. Then ${V}=(D[Y,{T}]/({h},{g}))_{{s}{s}'}$ is smooth over $D$. We claim that $I\subset ({h},{g})D[Y,{T}]_{{s}{s}'{s}''}$ for some other ${s}''\in 1+(d,{T})D[Y,{T}]$. Indeed, we have $PI\subset ({h},{g})D[Y,{T}]_{s}$ and so $P(y'+{s}^{-1}d^e\sum_{j=1}^qG_j(y')T_j)I\subset ({h},{g})D[Y,{T}]_{s}$. Since  $P(y'+{s}^{-1}d^e\sum_{j=1}^qG_j(y')T_j)\in P(y')+d^e({T})D[Y,T]_s$ we get \\ $P(y'+{s}^{-1}d^e\sum_{j=1}^qG_j(y')T_j)=d{s}''$ for some ${s}''\in 1+({T})D[Y,{T}]_s$. It follows that ${s}''I\subset (({ h},{g}):d)D[Y,{T}]_{{s}{s}'}$. Thus ${s}''IV\subset (0:_{V}d)\cap d^eV=0$ because $(0:_{V}d)\cap d^eV=0$, and $V$ is flat over $D$ and so over $A$.  This shows our claim.  It follows that
   $I\subset ({h},{g})D[Y,{T}]_{{s}{s}'{s}''}$. Thus ${E}_{{s}{s}'{s}''}\cong {V}_{{s}''} $ is a $B$-algebra which is also standard smooth over $D$ and $A$.

 As $\omega({s})\equiv 1$ modulo $d$ and ${\psi}({s}'),{\psi}({s}'')\equiv 1$ modulo $(d,{t})$, $d,{t}\in \mm A'$ we see that $\omega({s}),{\psi}({s}'), {\psi}({s}'')$ are invertible because  $A'$ is local. Thus ${\psi}$ (and so $v$) factors through the standard smooth $A$-algebra $B'={E}_{{s}{s}'{s}''}$.
 \end{proof}

{\bf Proof of Theorem \ref{m}}

 We choose   $\gamma_1, \gamma_2, \ldots , \gamma_m \in v(H_{B/A})A'\cap A$ such that $\gamma_k$ for $k \in [m]$ is a system of parameters  in $A$, and $\gamma_k=\sum_{i=1}^qv(b_i)z_i^{(k)}$, where $z_i^{(k)}\in A'$, $b_i\in H_{B/A}$. Set $B_0=B[Z^{(1)}, \ldots , Z^{(m)}]/(f^{(1)}, \ldots , f^{(m)}) $, where $ f^{(k)}=-\gamma_k +\sum_{i=1}^qb_iZ_i^{(k)}\in B[Z^{(k)}]$, $Z^{(k)}=(Z_1^{(k)},\ldots,Z_q^{(k)})$,
and let $v_0:B_0\to A'$ be the map of $B$-algebras given by $Z^{(k)}\to z^{(k)}$.
 Changing $B$ by $B_0$ we may suppose that $\gamma_k\in H_{B/A}$.

As in \cite{PP} we need the following lemma.

 \begin{Lemma}\label{lemma}
 \begin{enumerate}
   \item (\cite[Lemma 3.4]{P0}) \label{e1} Let $B_1$ be the symmetric algebra $S_B(I/I^2)$ of $I/I^2$ over\footnote{Let $M$ b e a finitely represented $B$-module and $B^m\xrightarrow{(a_{ij})} B^n\to M\to 0$ a presentation then $S_B(M)=B[T_1, \ldots, T_n]/J$ with $J=(\{\sum\limits^n_{i=1} a_{ij} T_i\}_ {j=1, \ldots, m})$.} $B$. Then $H_{B/A}B_1\subset H_{B_1/A}$ and  $(\Omega_{B_1/A})_{\gamma}$ is free over $(B_1)_{\gamma}$ for any $\gamma\in H_{B/A}$.
   \item  (\cite[Proposition 4.6]{S}) \label{e2}  Suppose that  $(\Omega_{B/A})_{\gamma}$ is free over $B_{\gamma}$. Let $I'=(I,Y')\subset A[Y,Y']$, $Y'=(Y'_1,\ldots,Y'_n)$. Then $(I'/I'^2)_{\gamma}$ is free over $B_{\gamma}$.
   \item (\cite[Corollary 5.10]{P2}) \label{e3} Suppose that $(I/I^2)_{\gamma}$ is free over  $B_{\gamma}$. Then a power of $\gamma$ is in  $ ((g):I)\Delta _g$ for some $g=(g_1,\ldots g_r)$, $r\leq n$ in $I$.
 \end{enumerate}
\end{Lemma}

Using $(1)$ of Lemma \ref{lemma} we reduce our proof to the case when $\Omega_{B_{\gamma_k}/A}$ for all $k \in[m]$ are free over $ B_{\gamma_k}$ respectively.

Let $B_1$ be given by Lemma \ref{e1}. The inclusion $B\subset B_1$ has a retraction $w$ which maps $I/I^2$ to zero. For the reduction we change $B,v$ by $B_1,vw$.

Using $(2)$ of Lemma \ref{lemma} we may reduce to the case when  $(I/I^2)_{\gamma_k}$ is free over $ B_{\gamma_k}$ for all $k \in [m]$.

Since $\Omega_{B_{\gamma_k}/A}$ is free over $ B_{\gamma_k}$ we see using Lemma \ref{e2} that changing $I$ with $(I,Y')\subset A[Y,Y']$ we may suppose that $(I/I^2)_{\gamma_k}$ is free over $ B_{\gamma_k}$.

Now using Using $(3)$ of Lemma \ref{lemma} we will reduce further to the case when a power  of $\gamma_k$ is in  $ ((f^{(k)}):I)\Delta _{f^{(k)}}$ for some $f^{(k)}=(f_1^{(k)},\ldots f_{r_k}^{(k)})$, $r_k\leq n$ from $I$.

We reduced to the case when $(I/I^2)_{\gamma_k}$ is free over $ B_{\gamma_k}$. Then it is enough to  use Lemma \ref{e3}.

Replacing $B_1$ by $B$ we may assume that a power $d_k$ of $\gamma_k$ for all $k \in [m]$ has the form $d_k\equiv P_k= \sum_{i=1}^{q_k}M_i^{(k)}L_i^{(k)}\ \mbox{modulo}\ I$,
for some $r_k\times r_k$ minors $M_i^{(k)}$
of $(\partial f^{(k)}/\partial Y) $ and $L_i^{(k)}\in ((f^{(k)}):I)$.

The Jacobian matrix $(\partial f^{(k)}/\partial Y)$ can be completed with $(n-r_k)$ rows from $A^n$ obtaining a square $n$ matrix $H_i^{(k)}$ such that $\det H_i^{(k)}=M_i^{(k)}$.

This is easy using just the integers $0,1$.
Set $d=d_m$, $f=f^{(m)}$, $r=r_m$, $q=q_m$, $M_i=M_i^{(m)}$, $N_i=N_i^{(m)}$,  $\bar A=A/d^{2e+1}$, ${\bar B}={\bar A}\otimes_AB$, ${\bar A}'=A'/(d^{2e+1}A')$, ${\bar v}={\bar A}\otimes_Av$. Then we have $d\equiv \sum_jM_jN_j$ modulo $I$.
Now we will use the induction on $m$.
\vskip 0.3 cm
\textbf{Case I:}  $m=0$
\vskip 0.3 cm
If $m=0$ then $A$ and $A'$ are Artinian local rings and $u :
A \rightarrow A'$ is a regular morphism. Then we are done by Corollary 3.3  \cite{P0}.
\vskip 0.3 cm
\textbf{Case II:}  $m>0$
\vskip 0.3 cm
Suppose by the induction hypothesis that  we have  a standard smooth $\bar{A}$-algebra $\bar{D}\cong (\bar A[Z]/(\bar g))_{\bar h \bar M }$, for $ Z=(Z_1,\ldots,Z_p), \bar  g=(\bar g_1, \ldots ,\bar g_q)$ with $q\leq p$, $\bar h\in \bar A[Z]$ and $\bar M$  a $q \times q$-minor of $(\frac{\partial \bar g}{\partial Z})$, such that the map $\bar{v}:\bar{B} \rightarrow \bar{A'}$ factors through  $\bar{D}$, let us say $\bar v$ is the composite map  $\bar{B} \rightarrow \bar {D} \xrightarrow{\bar \omega} \bar{A'}$.

Now let $g \in A[Z]^q$ be a lifting of $\bar g$ and $M$ the $q \times q$-minor of $(\frac{\partial g}{\partial Z})$ corresponding $\bar M$. Take $h \in A[Z]$ such that $h$ lifts $\bar h$.
 Then $D\cong (A[Z]/(g))_{h M}$ is a  standard smooth $A$-algebra and by the \emph{Implicit Function Theorem} the map ${\bar \omega}$ can be lifted to $\omega:D \to A'$ since $A'$ is Henselian. It follows that
$\Im v\subset \Im \omega + d^{2e+1}A'$. Applying Proposition \ref{proposition} we get a $B$-algebra $C$ smooth over $A$ such that $v$ factors through $C$, $B \rightarrow C \rightarrow A' $.\\

\section{A uniform  N\'eron Desingularization}

 Let $u:A\to A'$ be a regular morphism of Cohen-Macaulay local rings of dimension $m$. Suppose that  the maximal ideal $\mm$ of $A$ generates the maximal ideal of $A'$, $A'$ is Henselian and $A$, $A'$ have the same completions.

  Let $B=A[Y]/I$, $Y=(Y_1,\ldots,Y_n)$, and for $i \in [m]$ let $f^{(i)}=(f_1^{(i)},\ldots,f_{r_i}^{(i)})$, $r_i\leq n$ be a system of polynomials from $I$. Let $M_i$ be an $r_i \times r_i$-minor of the Jacobian matrix $(\partial f^{(i)}/ \partial Y)$ and $N_i \in ((f^{(i)}):I)$, $P_i=N_i M_i$.  Let  $v:B \to A'/\mm^{3k+c}A'$ be an $A$-morphism for some $k,c \in \mathbb{N}$. Suppose that $v(N_1M_1, \ldots , N_mM_m)A'/\mm^{3k+c}A' \supset \mm^kA'/\mm^{3k+c}A'$. Let $y' \in  A^n$ be a lifting of $v(Y)$ to $A$ and let $d_i=P_i(y')$. Then $(d_1, \ldots , d_m)A' /\mm^{3k+c}A'\supset \mm^kA'/\mm^{3k+c}A'$. Note that $\mm^k \subset (d_1, \ldots , d_m)A + \mm^{3k+c} \subset (d_1, \ldots , d_m)A + \mm^{3(3k+c)+c} \subset \ldots$. Thus $\mm^k \subset (d_1, \ldots , d_m)A $ and it follows that $(d_1, \ldots , d_m)A' \supset \mm^kA'$ . Since $A$ is Cohen-Macaulay we get $d=\{d_1, \ldots , d_m\}$ regular sequence in $A$. Note that $(d_1, \ldots , d_m)$ is the ideal corresponding to $v(P_1, \ldots , P_m)A'$ by the isomorphism $A/\mm^{3k+c} \cong A' /\mm^{3k+c}A'$.

  \begin{Theorem}\label{unith}
There exists a $B$-algebra $C$ which is standard smooth over $A$ with the following properties.
\begin{enumerate}
  \item
  Every $A$-morphism $v':B\to A'$ with $v'\equiv v \ \mbox{modulo}\ (d_1^3, \ldots , d_m^3) A'$ factors through $C$.
  \item
 Every $A$-morphism $v':B\to A'$ with $v'\equiv v \ \mbox{modulo}\ \mm^{3k}A'$ factors through $C$.
  \item
    There exists an $A$-morphism $w:C\to A'$ which makes the following diagram commutative
 $$
  \begin{xy}\xymatrix{B \ar[d] \ar[r] & C  \ar[r]^{w} & A'\ar[d]\\
  A/\mm^{3k+c} \ar[r] &  A/\mm^c \ar[r]   & A'/\mm^cA' }
  \end{xy}
  $$
\end{enumerate}
  \end{Theorem}

\begin{proof}
Let $v':B\to A'$ be an $A$-morphism with $v'\equiv v \ \mbox{modulo}\ (d_1^3, \ldots , d_m^3) A'$. We apply induction on $m$.
\vskip 0.3 cm
\textbf{Case I:}  $m=1$
\vskip 0.3 cm
If $m=1$ then $A$ and $A'$ are Noetherian local rings of dimension $1$ and $u :
A \rightarrow A'$ is a regular morphism. Then we are done by Theorem 2  \cite{KPP}, with $e=1$.
\vskip 0.3 cm
\textbf{Case II:}  $m>1$
\vskip 0.3 cm
Now let $\bar A =A/(d_1^3, \ldots , d_{m-1}^3)$, and consider the map $\bar{v'}={\bar A} {\otimes}_{A} v': \bar{B}={\bar A} {\otimes}_{A}B \to \bar{A}'={\bar A} {\otimes}_{A}A'$. By the induction hypothesis there exists a standard smooth algebra $\bar{D}\cong (\bar A[Z]/(\bar g))_{\bar h \bar M }$, for $ Z=(Z_1,\ldots,Z_p), \bar  g=(\bar g_1, \ldots ,\bar g_q)$ with $q\leq p$, $\bar h\in \bar A[Z]$ and $\bar M$  a $q \times q$-minor of $(\frac{\partial \bar g}{\partial Z})$, such that the map $\bar{v'}$ factors through $\bar{D}$, say $\bar {v'}$ is the composite map $\bar B \to \bar D \xrightarrow{\bar{\omega '}} A'$.

Now let $g \in A[Z]^q$ be a lifting of $\bar g$ and $M$ the $q \times q$-minor of $(\frac{\partial g}{\partial Z})$ corresponding $\bar M$. Take $h \in A[Z]$ such that $h$ lifts $\bar h$.
 Then $D\cong (A[Z]/(g))_{hM}$ is a  standard smooth $A$-algebra and by the \emph{Implicit Function Theorem} the map ${\bar {\omega'}}$ can be lifted to $\omega':D \to A'$ since $A'$ is Henselian. It follows that
$\Im v'\subset \Im \omega' + d_m^3A'$. Applying Proposition \ref{proposition} (with $e=1$) we get a $B$-algebra $C$ standard smooth over $A$ such that $v'$ factors through $C$. This proves (1) which obviously implies (2).\\

 Now  for (3) take the map ${\hat w}:C\cong (D[Y,T]/(I,g,h))_{ss'}\to A'/\mm^cA'$ given by $(Y,T)\to (y',0)$. Then the composite map $B\to C\xrightarrow{\hat w} A'/\mm^c A'$
 is lifted by $v$. Since $C$ is standard smooth, we may lift $\hat w$ to an $A$-morphism $w:C\to A'$ by the Implicit Function Theorem. Clearly, $w$ makes the above diagram commutative.
\end{proof}

\begin{Example} (Rond) \label{e}{\em Let $k$ be a field, $A=k[[x]]$, $x=(x_1,x_2,x_3)$, $B=A[Y]/(f)$, $Y=(Y_1,\ldots,Y_4)$, $f=Y_1Y_2-Y_3Y_4$. Then $\Delta_f=H_{B/A}=(Y)$. Let $p\in {\bf N}$ and set $y'_1=x_1^p$, $y'_2=x_2^p$, $y'_3=x_1x_2-x_3^p$. Then there exists $y'_4\in A$ such that $f(y')\equiv 0$ modulo $(x)^{p^2}$. It follows that $d_1=x_1^p$, $d_2=x_2^p$ and $d_3=x_3^{p^2}$ belongs to $\Delta_f(y') $ because $x_1^px_2^p-x_3^{p^2}=y'_3(x_1^{p-1}x_2^{p-1}+x_1^{p-2}x_2^{p-2}x_3^p+\ldots x_3^{p^2} $). For $k=2p+p^2-2$ we have $(x)^k\subset (d_1,d_2,d_3)\subset H_{B/A}(y')$. If $f(y')\equiv 0$ modulo $(x)^{3k+p+1}$ then by Theorem \ref{unith} (3) we could get $y\in A^4$ such that $f(y)=0$ and $y\equiv y'$ modulo $(x)^{p+1}$. But this is not the case, since $f(y')\equiv 0$ modulo $(x)^{p^2}$  and we cannot apply the quoted theorem. Thus it is not a surprise that \cite[Remark 4.7]{R} says that there exist no $y\in A^4$ such that $f(y)=0$ and $y\equiv y'$ modulo $(x)^{p+1}$.}
\end{Example}

\begin{Corollary} (Elkik) Let $(A,\mm)$ be a Cohen-Macaulay Henselian local ring of dimension $m$ and $B=A[Y]/I$, $Y=(Y_1,\ldots,Y_n) $ an $A$-algbra of finite type. Then for every $k\in \bf N$ there exist two integers
$m_0,p\in \bf N$ such that if $y'\in A^n$ satisfies $m^k\subset H_{B/A}(y')$ and $I(y')\equiv 0$ modulo $\mm^m$ for some    $m>m_0$ then there exists $y\in A^n$ such that $I(y)=0$ and $y\equiv y'$ modulo $\mm^{m-p}$.
\end{Corollary}

\begin{proof} Suppose that $A'=A$. In the notation of Theorem \ref{unith} given $k$ set $m_0=p=3k$ and  suppose that $y'\in A^n$ satisfies  $m^k\subset H_{B/A}(y')$ and $I(y')\equiv 0$ modulo $\mm^m$ for some    $m>m_0$. Let $v:B\to A/\mm^m$ be given by $Y\to y'$. Set $c=m-p$. By  Theorem \ref{unith} there exists a smooth
$A$-algebra $C$ and a map $w:C\to A$ which makes the above digram commutative. Let $y$ be the image of $Y$ by the composite map $B\to C\xrightarrow{w} A$. Then  $I(y)=0$ and $y\equiv y'$ modulo $\mm^c=\mm^{m-p}$.
\hfill\ \end{proof}

\begin{Corollary}\label{c1} With the assumptions and notation of the  Theorem \ref{unith}, let $\rho:B\to C$ be the structural algebra map. Then $\rho$ induces bijections $\rho^*$ given by $\rho^*(w)=w\circ \rho$, between
 \begin{enumerate}
   \item $\{w\in \Hom_A(C,A'):w\circ \rho\equiv v \ \mbox{modulo}\ (d_1^{3}, \ldots , d_m^3)A'\}$ and
 $\{v'\in \Hom_A(B,A'):v'\equiv v \ \mbox{modulo}\ (d_1^{3}, \ldots , d_m^3)A'\}$
   \item $\{w\in \Hom_A(C,A'):w\circ \rho\equiv v \ \mbox{modulo}\ \mm^{3k}A'\}$ and
 $\{v'\in \Hom_A(B,A'):v'\equiv v \ \mbox{modulo}\ \mm^{3k}A'\}$
 \end{enumerate}
 \end{Corollary}

 \begin{proof} We will use induction on the dimension of $A$.
 \begin{enumerate}
 \item \textbf{Case I:}  $m=1$
\vskip 0.3 cm
If $m=1$ then $A$ and $A'$ are Cohen-Macaulay local rings of dimension $1$. Then we are done by Corollary $8$ \cite{KPP}.
\vskip 0.3 cm
\item \textbf{Case II:}  $m>1$
\vskip 0.3 cm
 By Theorem \ref{unith}, (1), $\rho^*$ is surjective. Let Now let $\bar A =A/(d_1^3, \ldots , d_{m-1}^3)$, and the map $\bar{v'}={\bar A} {\otimes}_{A} v': \bar{B}={\bar A} {\otimes}_{A}B \to \bar{A}'={\bar A} {\otimes}_{A}A'$. By the induction hypothesis there exists a standard smooth algebra $\bar{D}$ such that the maps $\bar w$ and $\bar w'$ restricted to $\bar D$ coincide. This implies that $w\mid_D$ and $w'\mid_D$ lift the same map $\bar w\mid_{\bar D}$. Thus $w\mid_D=w'\mid_D$ by uniqueness in the Implicit Function Theorem.

 By construction $C=E_{ss'}$, $E= D[Y,T]/(I,g,h)$ and
 $H_m(y')(w(Y)-w'(Y))\equiv d_m^2(w(T)-w'(T))\ \mbox{modulo}\ h$. Thus  $d_m^2(w(T)-w'(T))=0$ and so $w|_E=w'|_E$ because $d_m$ is regular in $A'$ since $d_m$ is regular in $A$ and $u$ is flat. It follows that  $w=w'$

   \item Apply Theorem \ref{unith} (2) for the surjectivity. The injectivity follows from above.

 \end{enumerate}

\hfill\ \end{proof}

\begin{Corollary}\label{c2} With the assumptions and notation of the above Corollary, the following statements hold:
\begin{enumerate}
  \item  If there exists an $A$-morphism ${\tilde v}:B\to A'$
 with ${\tilde v}\equiv v\ \mbox{modulo}\ (d_1^{3}, \ldots , d_m^3)A' $, then
 there exists a unique $A$-morphism ${\tilde w}:C\to A'$ such that ${\tilde w}\circ \rho={\tilde v}$.
  \item  If there exists an $A$-morphism ${\tilde v}:B\to A'$
 with ${\tilde v}\equiv v\ \mbox{modulo}\ \mm^{3k}A' $, then
 there exists a unique $A$-morphism ${\tilde w}:C\to A'$ such that ${\tilde w}\circ \rho={\tilde v}$.
\end{enumerate}
  \end{Corollary}
  For the proof take ${\tilde w}={\rho^*}^{-1}(\tilde v)$, where $\rho^*$ is defined in the  Corollary \ref{c1}.
\\
 By  construction, $C$ has the form $(D[T]/(g))_{Mh}$, where $M=\det (\partial g_i/\partial T_j)_{i,j\in [r]} $ and  $h=s'\in A[T]$ satisfies ${\tilde w}(h)\not \in \mm A'$. 

\begin{Lemma}\label{lem} There  exist canonical bijections

\begin{enumerate}
  \item $ (d_1^{3}, \ldots , d_m^3)A'^{n-r}\to  \{w'\in \Hom_A(C,A'):w'\equiv {\tilde w} \ \mbox{modulo}\ (d_1^{3}, \ldots , d_m^3)A'\}.$
  \item $ \mm^{3k} A'^{n-r}\to  \{w'\in \Hom_A(C,A'):w'\equiv {\tilde w} \ \mbox{modulo}\ \mm^{3k}A'\}.$
\end{enumerate}
\end{Lemma}

\begin{proof}
Note that $ \{w'\in \Hom_A(C,A'):w'\equiv {\tilde w} \ \mbox{modulo}\ (d_1^{3}, \ldots , d_m^3)A'\}$ is in bijection with the set of all $t\in A'^n$ such that $g(t)=0$ and $t\equiv {\tilde w}(T)$ modulo $(d_1^3,\ldots,d_m^3) A'^n$.

Set $V=(T_1,\ldots,T_r)$, $Z=(T_{r+1},\ldots,T_n)$.
Thus $ g(U, w'(Z))=0$ has a unique solution  (namely $U= w'(V)$) in ${\tilde  w }(Z)+(d_1^{3}, \ldots , d_{m}^3) A'^{n-r}$ by the Implicit Function Theorem.
Consequently, $w'(V)$ is uniquely defined by $w'(Z)$, that is by the restriction $w'\mid_{ A[Z]}$.

Therefore, $ \{w'\in \Hom_A(C,A'):w'\equiv {\tilde w} \ \mbox{modulo}\ (d_1^{3}, \ldots , d_m^3)A'^n\}$ is in bijection with $\{w''\in \Hom_A(A[Z],A'):w''\equiv {\tilde w}|_{A[Z]} \ \mbox{modulo}\ (d_1^{3}, \ldots , d_m^3)A'^{n-r}\}$, the latter set being in bijection with ${\tilde w}(Z)+(d_1^{3}, \ldots , d_m^3)A'^{n-r}$, that is with $(d_1^{3}, \ldots , d_m^3)A'^{n-r}$.
 The proof of (2) goes similarly.
\hfill\ \end{proof}

\begin{Theorem} \label{arc1} With the assumptions and notation of Corollary \ref{c1} there exist canonical bijections
\begin{enumerate}
  \item $$(d_1^{3}, \ldots , d_m^3)A'^{n-r}\to \{v'\in \Hom_A(B,A'):v'\equiv v \ \mbox{modulo}\ (d_1^{3}, \ldots , d_m^3)A' \}.$$
  \item $$\mm^{3k}A'^{n-r}\to \{v'\in \Hom_A(B,A'):v'\equiv v \ \mbox{modulo}\ \mm^{3k}A' \}.$$
\end{enumerate}
 \end{Theorem}
For the proof apply Corollary \ref{c1} and the above lemma.

\section{Algorithms}
\vskip 0.5 cm
In this section we present the algorithms corresponding to the results of Sections 1 and 2.
We will use in our algorithm for uniform desingularization the following algorithm for the one dimensional case (cf. \cite{KPP}):
\vskip0.5 cm
\newpage
\begin{algorithm}[ht]\label{alg:UniformNeronDesing}
\begin{algorithmic}[1]
\REQUIRE $A, B, v, k, f, N$ given by the following data. $A=K[x]_{(x)}/J$, $J=(h_1, \ldots , h_{p'})$, $h_i \in K[x]$, $x=(x_1, \ldots , x_t)$, $\dim(A)=1$, $K$ a field. $B=A[Y]/I$, $I=(g_1, \ldots , g_l)$, $g_i \in K[x,Y]$, $Y=(Y_1, \ldots , Y_n)$, integer $k,c$, $f=(f_1, \ldots , f_r), f_i \in I$, $v:B \to A/m^c_A$ defined by $y' \in K[x]^n$, $N \in (f_1, \ldots , f_r):I.$
\ENSURE $(D, \pi)$ given by the following data. $D=(A[Z]/(g))_{hM}$ a standard smooth algebra, $Z=(Z_1, \ldots , Z_p)$, $g=(g_1, \ldots , g_q)$, $q\leq p$, $h\in A[Z]$, $M$ a $q \times q$ minor of $(\partial g/\partial Z)$, $\pi: B \to D$ given by $\pi(Y)$ and factorizing $v$ or the message ``$y', N, f_1, \ldots , f_r$ are not well chosen."
\newline
  \STATE Compute $M=\det((\partial f_i/\partial Y_j)_{i,j\in [r]})$, $P:=NM$ and $d:=P(y')$
  \STATE $f:= (f_1,\ldots,f_r)$
  \STATE Compute $e$ such that $(0:_Ad^e)=(0:_Ad^{e+1})$
  \IF{ $I(y')\nsubseteq (x)^{(2e+1)k} +J$ or $(x)^k \nsubseteq (d)+J$}
  \RETURN ``$y'$, $N$, $(f_1,\ldots,f_r)$ are not well chosen''
  \ENDIF
  \STATE Complete $(\partial f_i/\partial Y_j)_{i \leq r}$ by $(0| (\mbox{Id}_{n-r}))$ to obtain a square matrix $H$
  \STATE Compute $G'$ the adjoint matrix of $H$ and $G:=NG'$
  \STATE $h=Y-y'-d^e G(y')T,\  T=(T_1,\ldots, T_n)$
  \STATE Write $f(Y)-f(y')=  \sum_jd^e\partial f/\partial Y_j(y') G_j(y')T+d^{2e}Q$
  \STATE Write $f(y')=d^{e+1}a$
  \FOR{ $i=1$ to $r$}
  \STATE $g_i=a_i+T_i+d^{e-1}Q_i$
  \ENDFOR
  \STATE $E:=A[Y,T]/(I,g,h)$
  \STATE Compute $s$ the $r\times r$ minor defined by the first $r$ columns of $(\partial g/\partial T) $
  \STATE Write $P(y'+ d^eG(y')T)=ds'$
  \RETURN $E_{ss'}$.
\newline
\end{algorithmic}
\caption{UniformNeronDesingularizationDim1}
\end{algorithm}

\vskip 0.5 cm
Next we present the algorithm for uniform desingularization for the higher dimensional case.
\begin{algorithm}[ht]\label{algUniformNeron}
\begin{algorithmic}[1]
\REQUIRE $A, B, v, k, f, N$ given by the following data. 
$A=K[x]_{(x)}/J$, $J=(h_1, \ldots , h_{p'})$, $h_i \in K[x]$, $x=(x_1, \ldots , x_t)$, $\dim(A)=m$,
 $K$ a field. $B=A[Y]/I$, $I=(g_1, \ldots , g_l)$, $g_i \in K[x,Y]$, $Y=(Y_1, \ldots , Y_n)$, integer $k,c$, $f=(f_1, \ldots , f_r), f_i \in I$, $v:B
 \to A/\mm^c_A$ defined by $\bar{y} \in K[x]^n$, for $i \in [m]$, $f^{(i)}=(f^{(i)}_1, \ldots , f^{(i)}_{r_i})$,  $r_i \leq n$, $f^{(i)}_j\in I$, $N_i\in
 ((f^{(i)}):I)$, $M_i$ an $r_i \times r_i$ minor of $(\partial f^{(i)}/\partial Y)$.
\ENSURE $D, \pi$ given by the following data. $D=(A[Z]/(g))_{hM}$ a standard smooth, $Z=(Z_1, \ldots , Z_p)$, $g=(g_1, \ldots , g_q)$, $q\leq p$, $h\in A[Z]$, $M$ a $q \times q$ minor of $(\partial g/\partial Z)$, $\pi: B \to D$ given by $\pi(Y)=y'$ and factorizing $v$ or the message ``$y', f^{(i)}, N_i$ are not well chosen."
\newline
  \FOR{$i\in[m]$} \STATE $P_i:=M_iN_i$; $d_i:=P_i(\bar{y})$
  \ENDFOR
  \IF{$I(\bar{y})\nsubseteq (x)^{3k}+J$ or $(x)^k\nsubseteq (d_1, \ldots , d_m)+J$}
   \RETURN $"\bar{y}, f^{(i)}, N_i$ are not well chosen."
   \ENDIF
  \IF{$m=1$} 
  \RETURN UniformNeronDesingularizationDim1$(A, B, v, k, f^{(1)}, N_1)$
\ENDIF
  \STATE $\bar{A}:=A/(d_1^3, \ldots , d_{m-1}^3)$, $\bar{v}:=\bar{A}\otimes_A v$, $\bar{B}:=\bar{A}\otimes_A B$
  \STATE $(\bar{D},\bar{\pi}):= $ UniformNeronDesingularization $(\bar{A}, \bar{B}, \bar{v}, k, f^{(m)}, N_m,M_m)$\\
   $\bar{D}=(\bar{A}[Z]/(\bar{g}))_{\bar{h}\bar{M}}$, $Z=(Z_1, \ldots , Z_p)$, $g=(g_1, \ldots , g_q)$, $g_i \in K[x,Z]$\\
   $\bar{g}=g$ mod $(d_1^3, \ldots , d_{m-1}^3)$, $h \in k[x,Z]$, $\bar{h}=h$ mod $(d_1^3, \ldots , d_{m-1}^3)$, $M$ a $q \times q$ minor of $(\partial g/\partial Z)$, $\bar{M}= M$ mod $(d_1^3, \ldots , d_{m-1}^3)$, $\bar{\pi}(Y)=y'$
  \STATE $D:=(A[Z]/(g))_{hM} $
  \STATE Complete the Jacobian matrix associated to $f^{(m)}$ by rows of $0$ and $1$ to obtain a square matrix $H_m$ with $\det(H_m)=M_m$
  \STATE Compute $G_m'$ the adjoint matrix of $H_m$ and $G_m:= N_mG_m'$
  \STATE $P:=M_mN_m, d:=P(y'), f=f^{(m)}$, $r=r_m$
  \STATE Write $P(y')=ds$ for some $s\in D, s\equiv 1$ mod $d$
  \STATE $h:=s(Y-y')-d\sum_{j=1}^{r} G_j(y')T_j, T_j=(T_1,\ldots,T_r, T_{j,r+1},\ldots,T_{j,n})$
  \STATE $p=\max\{\deg({f_i}^{(m)})\}$
  \STATE  Write $s^p(f(Y)-f(y'))\equiv \sum_{j'}s^{p-1}d(\partial f/\partial Y_{j'})(y')\sum_{j=1}^q G_{jj'}(y')T_{jj'}+d^2Q$
mod $h$.
\STATE Write $f(y')=d^2b$, $b\in dD^r$.
\FOR{$i\in [r]$}
\STATE $g_i:=s^pb_i+s^pT_i+Q_i$.
\ENDFOR
\STATE $E:=D[Y,T]/(I,g,h)$.
\STATE Compute $s'$ the $r\times r$-minor of $(\partial g/\partial T)$ given by the first $r$ variables of $T$.
\STATE Compute $s''$ such that $P(y'+s^{-1}d\sum_{j=1}^qG_j(y')T_j)=ds''$.
\STATE Define $\pi:B\to (D[Y,T]/(I,g,h)_{ss's''}$ by $\pi(Y_i)=y'_i$.
\RETURN $((D[Y,T]/(I,g,h)_{ss's''},\pi)$.
\newline
\end{algorithmic}
\caption{UniformNeronDesingularization}
\end{algorithm}

\begin{algorithm}[ht]\label{alg:NeronDesing}
\caption{NeronDesingularization}
\small
\begin{algorithmic}[1]
\REQUIRE $A$, $B$, $v$, $N$ given by the following data:
$A=k[x]_{(x)}/J$, $J=(h_1,\ldots,h_q)\subset k[x]$, $x=(x_1,\ldots,x_t)$, $k$ a field, $k'=Q(k[U]/\bar{J})$, $\bar{J}=(a_1,\ldots,a_r)\subset k[U]$,
$U=(U_1,\ldots,U_{t'})$ separable over $k$, $B=A[Y]/I$, $I=(g_1,\ldots,g_l)\subset k[x,Y]$, $Y=(Y_1,\ldots,Y_n)$, $v:B\to A'\subset K[[x]]/JK[[x]]$
an $A$-morphism, given by $\bar{y}=(\bar{y}_1,\ldots,\bar{y}_n)\in k[x,U]^n$ approximations mod $(x)^N$ of $v(Y_i)$, $K\supset k'$ a field.
\ENSURE $(E,\pi)$ given by the following data:
$E=(A[Z]/L)_{hM}$ standard smooth, $Z=(Z_1,\ldots,Z_p)$, $L=(b_1,\ldots,b_{q'})\subset k[x,Z]$, $h\in k[x,Z]$, $M$ $q\times q$-minor of $(\partial b_i/\partial Z_j)$, $\pi:B\to E$ an $A$-morphism given by $\pi(Y_1)=y'_1, \ldots, \pi(Y_n)=y'_n$ factorizing $v$, i.e. there exists $\omega:E\to A'$ with $\omega \pi=v$.
\newline
\STATE Compute $w:=(a_{i_1},\ldots,a_{i_p})$,  $\rho$ a $p\times p$-minor of $(\partial a_{i_{\nu}}/\partial U_j)$ such that $\rho\not \in \bar{J}$.
 Compute $\tau\in (w):\bar{J}$ such that $k[U]_{\rho\tau}/(a_1,\ldots,a_r)=k[U]_{\rho\tau}/(w)$, $D:=A[U]_{\rho\tau}/(w)$.
\IF{If $\dim(A)=0$}
\STATE compute $f=(f_1,\ldots,f_r)$ in $I$, $N\in (f):I$ and $M$ an $r\times r$-minor of $(\partial f_i/\partial Y_j)$ such that $B_{NM}$ is standard smooth,
 return $((D[Y]/I)_{NM}, Y)$.
\ENDIF
\STATE Compute $H_{B/A}=(b_1,..,b_q)$
and  $\gamma_1,..,\gamma_m \in H_{B/A}(\bar{y})$, system of parameters in $A$.
\FOR{$j \in [m]$}
\STATE write $\gamma_j=\sum_{i=1}^qb_i(\bar{y})\bar{y}_{n+i+(j-1)q}$ mod $(\gamma_1^2,...,\gamma_m^2)$, $\bar{y}_j\in k[x,U]$.
\ENDFOR
\FOR{$j \in [m]$}
\STATE $g_{l+j}:=-\gamma_j+\sum_{i=1}^q b_iY_{n+i+(j-1)q}$; 
\ENDFOR
\STATE $Y:=(Y_1,\ldots,Y_{n+mq})$;
$\bar{y}=(\bar{y}_1\ldots,\bar{y}_{n+mq})$; $I:=(g_1,\ldots,g_{l+m})$; $l:=l+m$; $n:=n+mq$, $B:=A[Y]/I$, $\gamma:=\gamma_m$
\STATE  $B=S_B(I/I^2)$, $v$ trivially extended.
Write $B:=A[Y]/I$, $Y=(Y_1,\ldots,Y_n)$; $Y:=(Y,Z)$;  $I:=(I,Z)$, $B:=A[Y]/I$,
$Z=(Z_1,\ldots,Z_n)$, $v$ trivially extended.
\STATE Compute $f=(f_1,\ldots,f_r)$ such that a power $d$ of $\gamma$ is in $((f):I)\Delta_f$.
\STATE  Compute $e$ such that $(0:d^e)=(0:d^{e+1})$ and  $p:=\max_i\{\deg f_i\}$.
\STATE Choose $r\times r$-minors $M_i$ of $(\partial f/\partial Y)$ and $N_i\in ((f):I)$ such that for $P:=\sum M_iN_i$
we have $d\equiv P$ mod $I$.
\STATE Complete the Jacobian submatrices of $(\partial f/\partial Y)$ corresponding to  $M_i$ by $n-r$ rows of $0$ and $1$ to obtain square matrices $H_i$ with
$\det H_i=M_i$.
\FOR{$j\in [m]$ } \STATE compute $G'_j$ the adjoint matrix of $H_j$ and $G_j:=N_jG'_j$
\ENDFOR
\STATE $\bar{A}:=A/(d^{2e+1})$; $\bar{B}:=\bar{A}\otimes_AB$, $\bar{v}:=\bar{A}\otimes v$.
\STATE $(\bar{E},\bar{\pi})=$NeronDesingularization$(\bar{A},\bar{B},\bar{v},N)$,
$\bar{E}=(\bar{A}[Z_1,\ldots,Z_p]/{\bar L})_{hM}$ standard smooth, $L=(b_1,\ldots,b_q)\subset k[x,Z]$, $\bar{L}=L$ mod $d^{2e+1}$, $h\in k[x,Z]$, $M$ $q \times q$-minor of $(\partial b/\partial Z)$, $\bar{\pi}:\bar{B}\to \bar{E}$ factorization of $\bar{v}$ given by $y'$ from $k[x,Z]^n$.
\STATE $D:=A[Z]_{hM}/L$, write $P(y')=ds$,  $f(y')=d^{e+1}b$, $b\in d^eD^r$, $s\in D$, $d | s-1$.
\STATE $h:=s(Y-y')-d^e\sum_{i=1}^q G_j(y') T_j$, $T_j=(T_1,\ldots, T_r,T_{j,r+1},\ldots, T_{j,n})$.
\STATE Write $s^p(f(Y)-f(y'))\equiv \sum_{j'}s^{p-1}d^e(\partial f/\partial Y_{j'})(y')\sum_j G_{jj'}(y')T_{jj'}+d^eQ$ mod $h$.
\FOR{$i\in [r]$} \STATE $g_i:=s^pb_i+s^pT_i+d^{e-1}Q_i$.
\ENDFOR
\STATE $E:=D[Y,T]/(I,g,h)$.
\STATE Compute $s'$ the $r\times r$-minor of $(\partial g/\partial T)$ given by the first $r$ variables of $T$.
\STATE Compute $s''$ such that $P(y'+s^{-1}d^e\sum_{j=1}^qG_j(y')T_j)=ds''$.
\STATE Define $\pi:B\to (D[Y,T]/(I,g,h)_{ss's''}$ by $\pi(Y_i)=y'_i$.
\RETURN $((D[Y,T]/(I,g,h)_{ss's''},\pi)$.
\newline
\end{algorithmic}
\end{algorithm}

\end{document}